\documentclass[11pt]{article}

\usepackage[T1]{fontenc}
\usepackage{lmodern}
\usepackage{amsmath,amssymb,amsthm}
\usepackage{microtype}
\usepackage{geometry}
\usepackage[hidelinks]{hyperref}
\hypersetup{
  pdftitle={Quadratic Reduction and Geometric Subsequences of Fibonacci and Lucas Numbers Modulo an Integer},
  pdfauthor={Marco Mantovanelli},
  pdfsubject={Quadratic reduction and geometric subsequences modulo an integer},
  pdfkeywords={Fibonacci numbers, Lucas sequences, Dickson polynomials, congruences}
}
\title{Quadratic Reduction and Geometric Subsequences\\
of Fibonacci and Lucas Numbers Modulo an Integer}
\author{Marco Mantovanelli\\[0.3em]
{\small Independent Researcher}\\
{\small \href{mailto:marco@mantovanelli.de}{\texttt{marco@mantovanelli.de}}}\\
{\small ORCID: \href{https://orcid.org/0009-0002-0631-293X}{\texttt{0009-0002-0631-293X}}}}
\date{\small Version 3}

\newtheorem{theorem}{Theorem}[section]
\newtheorem{lemma}[theorem]{Lemma}
\newtheorem{corollary}[theorem]{Corollary}
\newtheorem{proposition}[theorem]{Proposition}
\theoremstyle{remark}
\newtheorem{remark}[theorem]{Remark}
\newtheorem*{normalizationremark}{Remark}

\newcommand{\Aalg}{\mathcal{A}}

\begin{document}

\maketitle

\begin{abstract}
We determine when the terms of a Fibonacci recurrence at equally
spaced indices form a geometric progression modulo an integer.
Quadratic reduction over a commutative ring gives both the recurrence
for these terms and a matrix proof of Lucas's 1878 multiple-index
formulas for Fibonacci and Lucas numbers.

Let $W_{j+2}=W_{j+1}+W_j$ with integer initial values $W_0=a$ and
$W_1=b$, not both zero. Fix integers $r\geq0$ and $n\geq1$, set
$S_k=W_{r+kn}$, and write $h=\gcd(S_0,S_1)$. For an integer $M\geq2$,
there exists a residue $q$ modulo $M$ satisfying
\[
 S_{k+1}\equiv qS_k\pmod{M} \qquad\text{for every } k\geq0
\]
if and only if $M$ divides the positive integer
\[
 K_W(r,n)=\frac{|a^2+ab-b^2|F_n^2}{h}.
\]

For each admissible $M$, the multipliers form a single residue class
modulo $M/\gcd(M,h)$, giving exactly $\gcd(M,h)$ residues modulo $M$.
For the Fibonacci and Lucas sequences, respectively, the integer
$K_W(r,n)$ specializes to
\[
 K_F(r,n)=\frac{F_n^2}{F_{\gcd(r,n)}}\qquad\text{and}\qquad
 K_L(r,n)=\frac{5F_n^2}{\gcd(L_r,L_{r+n})}.
\]
\end{abstract}



\section{Introduction}

The expansions of $F_{nm}$ and $L_{nm}$ in powers of $L_n$ go back to
\'Edouard Lucas's work on the two sequences associated with a quadratic
equation~\cite[\S\S XII, XIV]{Lucas1878}.  Hoggatt and Lind later obtained
the Fibonacci expansion by means of compositions and generating
functions~\cite{HoggattLind1969}.  Gould traced the successive rediscoveries
and wrote both expansions in a form valid for either parity of the
step~\cite[Eqs.~(6), (9), (11.1)]{Gould1977}.  Section~\ref{sec:fibonacci}
identifies the formulas in Lucas's notation, including a misprinted index
in the 1878 text.

Mc~Laughlin~\cite[Eq.~(11)]{McLaughlin2004} derived the Fibonacci expansion
by computing $(A^n)^m$ for the Fibonacci matrix $A$ and attributed the
identity to a 2003 note by R.~C. Johnson.  Vorobtsov~\cite{Vorobtsov2026}
derived the Fibonacci, Lucas, and arbitrary-initial-value versions using
symmetric polynomials.  Here the same three identities follow by applying
different linear functionals to a single matrix equation.  The coefficients
are generic Lucas polynomials, equivalently Dickson polynomials after an
index shift~\cite{BallotWilliams2023,LidlMullenTurnwald1993}.  Earlier
matrix-power formulas include those of Blatz~\cite{Blatz1968} and
Williams~\cite{Williams1992}.  Matrix traces connect the Dickson and
Chebyshev descriptions~\cite{BrandiRicci2020,Pain2024}.  Background on
the classical sequences and polynomials can be found in
\cite{Koshy2001,MasonHandscomb2003,Vajda1989}.

The quadratic reduction and multiple-index identities in
Sections~\ref{sec:quadratic}--\ref{sec:fibonacci} are classical.
The contribution of Section~\ref{sec:congruences} is a classification for a
prescribed progression and arbitrary integral initial values.  For nonzero
initial pairs, it gives a positive integer whose divisors greater than one
are precisely the admissible moduli, together with the full residue class
of multipliers for each modulus.  The result includes initial terms that
are zero divisors modulo $M$.

\section{Standard quadratic reduction}
\label{sec:quadratic}

Let $T,D$ be indeterminates.  Define polynomials
$P_m(T,D)\in\mathbb Z[T,D]$ by
\begin{equation}
P_0(T,D)=0,\qquad P_1(T,D)=1,
\label{eq:Pinitial}
\end{equation}
and
\begin{equation}
P_{m+1}(T,D)=T P_m(T,D)-D P_{m-1}(T,D)
\qquad (m\geq1).
\label{eq:Prec}
\end{equation}
Their generating function is
\begin{equation}
\sum_{m\geq0}P_m(T,D)z^m
=\frac{z}{1-Tz+Dz^2}.
\label{eq:Pgf}
\end{equation}
Expanding the right-hand side as a formal power series gives
\begin{equation}
P_m(T,D)=
\sum_{i=0}^{\lfloor(m-1)/2\rfloor}
\binom{m-1-i}{i}T^{m-1-2i}(-D)^i
\qquad (m\geq1).
\label{eq:Pexplicit}
\end{equation}

We use the following standard reduction identity.

\begin{theorem}[Quadratic reduction]
\label{thm:quadratic-reduction}
Let $R$ be a commutative ring with identity, let $\Aalg$ be an associative
unital $R$-algebra, and regard scalars through the structure map
$R\to Z(\Aalg)$.  Suppose that $x\in\Aalg$ satisfies
\begin{equation}
x^2-tx+d\,1_{\Aalg}=0
\label{eq:quadratic-relation}
\end{equation}
for some $t,d\in R$.  Then, for every $m\geq1$,
\begin{equation}
x^m=P_m(t,d)x-dP_{m-1}(t,d)\,1_{\Aalg}.
\label{eq:quadratic-reduction}
\end{equation}
\end{theorem}

\begin{proof}
For $m=1$, Equation~\eqref{eq:quadratic-reduction} follows from
$P_1=1$ and $P_0=0$.  If it holds for $m$, then
\begin{align*}
x^{m+1}
&=P_m(t,d)x^2-dP_{m-1}(t,d)x\\
&=\bigl(tP_m(t,d)-dP_{m-1}(t,d)\bigr)x
   -dP_m(t,d)\,1_{\Aalg}\\
&=P_{m+1}(t,d)x-dP_m(t,d)\,1_{\Aalg},
\end{align*}
by Equation~\eqref{eq:Prec}.  The result follows by induction.
\end{proof}

\begin{remark}[Universal form]
Let $\overline X$ be the image of $X$ in
\[
\mathbb Z[T,D,X]/(X^2-TX+D).
\]
Then Equation~\eqref{eq:quadratic-reduction} is the polynomial identity
\[
\overline X^m=P_m(T,D)\overline X-DP_{m-1}(T,D).
\]
Every instance of Theorem~\ref{thm:quadratic-reduction} is obtained from this
identity by specialization.  No injectivity of the structure map
$R\to\Aalg$ is required.
\end{remark}

\section{Lucas, Dickson, and Chebyshev polynomials}
\label{sec:polynomials}

In the standard notation for Lucas sequences, $P_m(T,D)$ is the generic
Lucas sequence of the first kind with parameters $(T,D)$.  If
$E_r(X,A)$ denotes the Dickson polynomial of the second kind,
\[
E_r(X,A)=
\sum_{i=0}^{\lfloor r/2\rfloor}
\binom{r-i}{i}X^{r-2i}(-A)^i,
\]
then Equation~\eqref{eq:Pexplicit} says exactly that
\begin{equation}
P_m(T,D)=E_{m-1}(T,D)\qquad(m\geq1).
\label{eq:Dickson-second}
\end{equation}

For later use, define the companion sequence
\begin{equation}
Q_0(T,D)=2,\qquad Q_1(T,D)=T,
\qquad Q_{m+1}=TQ_m-DQ_{m-1}.
\label{eq:Qrec}
\end{equation}
It is the generic Lucas sequence of the second kind, equivalently the Dickson
polynomial of the first kind, and for $m\geq1$ satisfies
\begin{equation}
Q_m(T,D)=TP_m(T,D)-2DP_{m-1}(T,D).
\label{eq:QP}
\end{equation}
For $m\geq1$, its explicit form is
\begin{equation}
Q_m(T,D)=
\sum_{i=0}^{\lfloor m/2\rfloor}
\frac{m}{m-i}\binom{m-i}{i}
T^{m-2i}(-D)^i.
\label{eq:Qexplicit}
\end{equation}
Indeed, substituting Equation~\eqref{eq:Pexplicit} into
Equation~\eqref{eq:QP} shows that the coefficient of
$T^{m-2i}(-D)^i$ is
\[
\frac{m}{m-i}\binom{m-i}{i}
=\binom{m-1-i}{i}+2\binom{m-1-i}{i-1},
\]
where a binomial coefficient with lower index $-1$ is understood to be zero.
In particular, the rational-looking coefficients in
Equation~\eqref{eq:Qexplicit} are integers.

The relation with the classical Chebyshev polynomials of the second kind also
has a division-free polynomial form.  Only the subsequent expression in
terms of arbitrary parameters $t$ and $d$ requires invertibility assumptions.

\begin{proposition}[Chebyshev specialization]
\label{prop:chebyshev}
For every $m\geq1$, the division-free identity
\begin{equation}
P_m(2\Delta Z,\Delta^2)=\Delta^{m-1}U_{m-1}(Z)
\label{eq:chebyshev-division-free}
\end{equation}
holds in $\mathbb Z[\Delta,Z]$.  Consequently, let $S$ be a commutative ring,
let $t,d\in S$, and suppose that $2\in S$ is
invertible and that there is a unit $\delta\in S$ satisfying $\delta^2=d$.
Then
\begin{equation}
P_m(t,d)=\delta^{m-1}
U_{m-1}\!\left(\frac{t}{2\delta}\right)
\qquad (m\geq1),
\label{eq:chebyshev-specialization}
\end{equation}
where $U_r$ is the classical Chebyshev polynomial of the second kind.
\end{proposition}

\begin{proof}
The two sides of Equation~\eqref{eq:chebyshev-division-free} agree for
$m=1,2$.  The Chebyshev recurrence
$U_{r+1}(Z)=2ZU_r(Z)-U_{r-1}(Z)$ becomes, after multiplication by the
appropriate power of $\Delta$, precisely the recurrence
\eqref{eq:Prec} with $(T,D)=(2\Delta Z,\Delta^2)$.  The divided formula
follows on setting $\Delta=\delta$ and $Z=t/(2\delta)$.
\end{proof}

When these hypotheses are unavailable, Equations~\eqref{eq:Pexplicit} and
\eqref{eq:Dickson-second} remain valid without modification, including over
rings of characteristic $2$ and when $d$ is a zero divisor.

\section{Matrices and linear functionals}
\label{sec:matrices}

\begin{corollary}[Powers of a $2\times2$ matrix]
\label{cor:matrix-powers}
Let $R$ be a commutative ring with identity and let $M\in M_2(R)$.  Set
\[
t=\operatorname{tr}(M),\qquad d=\det(M).
\]
Then, for every $m\geq1$,
\begin{equation}
M^m=P_m(t,d)M-dP_{m-1}(t,d)I.
\label{eq:matrix-powers}
\end{equation}
\end{corollary}

\begin{proof}
The $2\times2$ Cayley--Hamilton identity over a commutative ring is
$M^2-tM+dI=0$.  Apply Theorem~\ref{thm:quadratic-reduction}.
\end{proof}

The scalar coefficients in Equation~\eqref{eq:matrix-powers} depend only on
the trace and determinant; the formula still contains the matrix $M$ itself.
This qualification avoids the misleading suggestion that $M^m$ is determined
by $t$ and $d$ alone.

\begin{corollary}[Linear-functional form]
\label{cor:linear-functional}
Under the hypotheses of Corollary~\ref{cor:matrix-powers}, let
$\lambda:M_2(R)\to R$ be $R$-linear.  Then
\begin{equation}
\lambda(M^m)=P_m(t,d)\lambda(M)
-dP_{m-1}(t,d)\lambda(I).
\label{eq:linear-functional}
\end{equation}
In particular,
\begin{equation}
\operatorname{tr}(M^m)=Q_m(t,d).
\label{eq:trace-powers}
\end{equation}
\end{corollary}

\begin{proof}
Apply $\lambda$ to Equation~\eqref{eq:matrix-powers}.  For
$\lambda=\operatorname{tr}$, use $\operatorname{tr}(I)=2$ and
Equation~\eqref{eq:QP}.
\end{proof}

\section{Fibonacci, Lucas, and generalized Fibonacci identities}
\label{sec:fibonacci}

Let
\[
F_0=0,\quad F_1=1,\qquad L_0=2,\quad L_1=1,
\]
and let both sequences satisfy $X_{r+1}=X_r+X_{r-1}$ for $r\geq1$.  For the Fibonacci
matrix
\[
A=\begin{pmatrix}1&1\\1&0\end{pmatrix},
\]
one has, for $n\geq1$,
\begin{equation}
A^n=\begin{pmatrix}F_{n+1}&F_n\\F_n&F_{n-1}\end{pmatrix},
\qquad
\operatorname{tr}(A^n)=L_n,
\qquad
\det(A^n)=(-1)^n.
\label{eq:fibonacci-matrix}
\end{equation}
Set $M=A^n$.  Then $A^{nm}=M^m$.

Throughout the remainder of this section, $m,n\geq1$.

\subsection{Fibonacci numbers}

Take $\lambda(X)=X_{12}$.  Since $\lambda(M)=F_n$ and
$\lambda(I)=0$, Corollary~\ref{cor:linear-functional} gives the compact
identity
\begin{equation}
F_{nm}=F_nP_m\!\left(L_n,(-1)^n\right).
\label{eq:Fcompact}
\end{equation}
Using Equation~\eqref{eq:Pexplicit}, this becomes
\begin{equation}
F_{nm}=F_n
\sum_{i=0}^{\lfloor(m-1)/2\rfloor}
\binom{m-1-i}{i}
L_n^{m-1-2i}(-1)^{i(n+1)}.
\label{eq:Fexpanded}
\end{equation}
Lucas~\cite[\S XII, p.~208, Eqs.~(74)--(76)]{Lucas1878} expresses the
quotient $U_{nr}/U_r$ of his first-kind sequence as a polynomial
$\psi_{n-1}$ in $V_r$ and $Q^r$.  His (76) prints $U_{2nr}/U_r$, with
the same extra factor of two in the indices on the left of (75).  The
correct quotient appears in his later book~\cite[\S 183, p.~318]{Lucas1891}.
Setting $(P,Q)=(1,-1)$ and replacing his $(n,r)$ by $(m,n)$
gives~\eqref{eq:Fexpanded}.  Gould~\cite[Eqs.~(6), (9)]{Gould1977}
records both the Fibonacci and general forms and notes the misprint.
The Fibonacci form is also Theorem~1 of Vorobtsov~\cite{Vorobtsov2026}.

Proposition~\ref{prop:chebyshev} now gives the correct parity-dependent form
over $\mathbb C$:
\begin{equation}
F_{nm}=
\begin{cases}
F_n U_{m-1}\!\left(L_n/2\right),& n\text{ even},\\[4pt]
F_n i^{m-1}U_{m-1}\!\left(L_n/(2i)\right),& n\text{ odd}.
\end{cases}
\label{eq:Fchebyshev}
\end{equation}
The factor for odd $n$ is essential.  At $n=1$, $m=3$, the even-step
expression would give $U_2(1/2)=0$, whereas $F_3=2$.

\subsection{Lucas numbers}

Taking traces in $A^{nm}=(A^n)^m$ and using
Equation~\eqref{eq:trace-powers} gives
\begin{equation}
L_{nm}=Q_m\!\left(L_n,(-1)^n\right).
\label{eq:Lcompact}
\end{equation}
Equivalently,
\begin{equation}
L_{nm}=
\sum_{i=0}^{\lfloor m/2\rfloor}
\frac{m}{m-i}\binom{m-i}{i}
L_n^{m-2i}(-1)^{i(n+1)}.
\label{eq:Lexpanded}
\end{equation}
Lucas~\cite[\S XIV, p.~212, Eq.~(86)]{Lucas1878} gives the expansion of
$V_{np}$ for general parameters.  With $p=m$ and parameters $(1,-1)$
this is~\eqref{eq:Lexpanded}.  The same sum appears in
Gould~\cite[Eq.~(11.1)]{Gould1977} and in
Vorobtsov~\cite[Theorem~2]{Vorobtsov2026}.

\subsection{Generalized Fibonacci sequences}

Let $R$ be a commutative ring with identity, regard the integer matrix $A$ as
an element of $M_2(R)$ through the canonical map $\mathbb Z\to R$, choose
$G_0,G_1\in R$, and define
\begin{equation}
G_{r+1}=G_r+G_{r-1}\qquad (r\geq1).
\label{eq:Grec}
\end{equation}
For $r\geq1$,
\[
G_r=G_1F_r+G_0F_{r-1}.
\]
Thus the $R$-linear functional
\[
\lambda_G(X)=G_1X_{12}+G_0X_{22}
\]
satisfies $\lambda_G(A^r)=G_r$ and $\lambda_G(I)=G_0$.  Applying
Corollary~\ref{cor:linear-functional} to $M=A^n$ yields
\begin{equation}
G_{nm}=G_nP_m\!\left(L_n,(-1)^n\right)
-(-1)^nG_0P_{m-1}\!\left(L_n,(-1)^n\right).
\label{eq:Gcompact}
\end{equation}
Expanding both polynomials and shifting the index in the second sum gives
\begin{align}
G_{nm}
={}&G_n\sum_{i=0}^{\lfloor(m-1)/2\rfloor}
\binom{m-1-i}{i}L_n^{m-1-2i}(-1)^{i(n+1)}
\notag\\
&+G_0\sum_{i=1}^{\lfloor m/2\rfloor}
\binom{m-1-i}{i-1}L_n^{m-2i}(-1)^{i(n+1)}.
\label{eq:Gexpanded}
\end{align}
Here and above, a sum whose lower limit exceeds its upper limit is understood
to be empty.
This is the form appearing as Theorem~3 in
Vorobtsov~\cite{Vorobtsov2026}.  Equation~\eqref{eq:Gcompact} shows that the
two sums arise from the two scalar terms in the universal matrix reduction.

\section{Geometric subsequences modulo an integer}
\label{sec:congruences}

Fix $r\geq0$ and $n\geq1$.  If $W$ satisfies the Fibonacci recurrence,
then $S_k=W_{r+kn}$ satisfies
\begin{equation}
 S_{k+2}=L_nS_{k+1}-(-1)^nS_k\qquad(k\geq0).
 \label{eq:step-recurrence}
\end{equation}
Indeed, $A^n$ advances the column $(W_{j+1},W_j)^{\mathsf T}$ by $n$
indices.  Its trace and determinant in~\eqref{eq:fibonacci-matrix} give
\eqref{eq:step-recurrence} by Cayley--Hamilton.  Recurrences on arithmetic
progressions are classical and extend to arbitrary
orders~\cite{BirmajerGilWeiner2016}.  Here we fix the sequence and the
progression, and ask for all $M\geq2$ for which there exists
$q\in\mathbb Z/M\mathbb Z$ such that
\begin{equation}
 S_{k+1}\equiv q S_k\pmod M\qquad(k\geq0).
 \label{eq:geometric}
\end{equation}
We call the sequence \emph{geometric modulo $M$} in this sense, allowing
every residue $q$ and including the zero sequence.

Ide and Renault~\cite[Theorem~2.1]{IdeRenault2012} classify moduli admitting
a sequence $1,q,q^2,\ldots$ that satisfies the original Fibonacci recurrence.
Carlip and Somer~\cite[Definition~2.1]{CarlipSomer2003} study first-order
subsequences modulo odd prime powers by choosing steps from a family of
restricted periods.  For an odd prime $p$ and $e\geq1$, their
\emph{restricted period} of $S$ modulo $p^e$ is the least $d>0$ for
which one multiplier takes $S_k$ to $S_{k+d}$ for every $k$.
Their recurrence is written $w_{k+2}=A w_{k+1}-B w_k$.  In
\eqref{eq:step-recurrence}, $A=L_n$ and $B=(-1)^n$, so $p\nmid B$.
If $p\nmid\gcd(S_0,S_1)$, the sequence belongs to their family and their
Proposition~2.13 applies~\cite[Defs.~2.1--2.2, pp.~138--139;
Prop.~2.13, p.~145]{CarlipSomer2007}.  Set $i=0$ and $j=c=1$ there.
The fundamental sequence starts with $0,1$, so it cannot have restricted
period one modulo $p$ and their exponent $\ell$ is zero.  Their second
exponent is at least $e$ exactly when $S$ has restricted period one modulo
$p^e$.  Since their determinant is $S_1^2-S_0S_2$, this is equivalent
to $p^e\mid S_0S_2-S_1^2$.

Here the sequence, starting index, and step are fixed.  We seek one
divisibility criterion valid for every modulus, including those sharing
a factor with both initial terms.
The following lemma uses the elementary two-row reduction of a system of
linear congruences~\cite{Smith1861}.

\begin{lemma}[A three-term test]\label{lem:modular}
Let $(S_k)_{k\geq0}$ be an integer sequence satisfying
$S_{k+2}=sS_{k+1}+tS_k$ for fixed $s,t\in\mathbb Z$.  Put
$(A_0,B_0,C_0)=(S_0,S_1,S_2)$ and $h=\gcd(A_0,B_0)>0$.
Then, for every integer $M\geq2$, $S$ is geometric modulo $M$ if and only if
\begin{equation}
 M\ \bigm|\ \frac{A_0 C_0-B_0^2}{h}.
 \label{eq:det-criterion}
\end{equation}
When this holds, there are exactly $\gcd(M,h)$ multipliers modulo $M$.
\end{lemma}

\begin{proof}
For a fixed $q$, the errors $S_{k+1}-qS_k$ satisfy the same recurrence
as $S$.  Thus \eqref{eq:geometric} is equivalent to the two congruences
$B_0-qA_0\equiv C_0-qB_0\equiv0\pmod M$.
Since $h$ divides $C_0=sB_0+tA_0$, write $(A_0,B_0,C_0)=h(a,b,c)$ and
set $D=ac-b^2$.  The two congruences imply
\[
 M\mid a(C_0-qB_0)-b(B_0-qA_0)=hD.
\]
Conversely, choose $u,v\in\mathbb Z$ with $ua+vb=1$ and put $q_0=ub+vc$.
Then
\[
 B_0-q_0A_0=-hvD,\qquad C_0-q_0B_0=huD.
\]
Hence $M\mid hD$ suffices.  Given this solution, $q_0+e$ also works
precisely when $M$ divides both $eA_0$ and $eB_0$, or equivalently when
$M\mid eh$.  All multipliers are therefore congruent to $q_0$ modulo
$M/\gcd(M,h)$, which gives the stated count.
\end{proof}

If $A_0=B_0=0$, the recurrence gives the zero sequence and every multiplier
works.  If $h>0$ and $A_0 C_0=B_0^2$, the lemma permits every modulus.

\begin{theorem}[Arbitrary integral initial values]\label{thm:modular}
Let $W_0=a$, $W_1=b$ and $W_{j+2}=W_{j+1}+W_j$ for $j\geq0$, where
$(a,b)\in\mathbb Z^2\setminus\{(0,0)\}$.  Fix $r\geq0$ and $n\geq1$, and
write $\Delta_W=a^2+ab-b^2$ and $h=\gcd(W_r,W_{r+n})$.  Then $h>0$ and
\begin{equation}
 K_W(r,n)=\frac{|\Delta_W|F_n^2}{h}
 \label{eq:general-modulus}
\end{equation}
is a positive integer.  For each $M\geq2$, the subsequence
$(W_{r+kn})_{k\geq0}$ is geometric modulo $M$ exactly when $M\mid K_W(r,n)$.
For each such $M$ the multipliers form one residue class modulo
$M/\gcd(M,h)$ and hence comprise $\gcd(M,h)$ residues modulo $M$.
\end{theorem}

\begin{proof}
To evaluate the determinant in Lemma~\ref{lem:modular}, form the two matrices
\[
 H_r=\begin{pmatrix}W_r&W_{r+1}\\ W_{r+1}&W_{r+2}\end{pmatrix},
 \qquad
 T_n=\begin{pmatrix}1&F_{n-1}\\ 0&F_n\end{pmatrix}.
\]
By the Fibonacci recurrence $H_{r+1}=
\left(\begin{smallmatrix}0&1\\1&1\end{smallmatrix}\right) H_r$, so
$\det H_r=(-1)^r\det H_0$.  The addition formula
$W_{j+n}=F_{n-1}W_j+F_nW_{j+1}$, applied first to the columns and then to
the rows of $H_r$, gives
\[
 T_n^{\mathsf T}H_rT_n=
 \begin{pmatrix}W_r&W_{r+n}\\W_{r+n}&W_{r+2n}\end{pmatrix}.
\]
Taking determinants yields
\begin{equation}
 W_rW_{r+2n}-W_{r+n}^2=(-1)^r F_n^2\det H_0 .
 \label{eq:sample-det}
\end{equation}
Here $\det H_0=a(a+b)-b^2=\Delta_W$.
Equation~\eqref{eq:sample-det} is the Fibonacci specialization of Horadam's
generalized Catalan identity~\cite[p.~171, Eq.~(4.5)]{Horadam1965}.
If $\Delta_W=0$, then $(2b-a)^2=5a^2$.  The irrationality of $\sqrt5$
forces $a=b=0$, contrary to the hypothesis.  Thus the determinant in
\eqref{eq:sample-det} is nonzero.  In particular $W_r$ and $W_{r+n}$
cannot both vanish, so $h>0$.  Lemma~\ref{lem:modular}, applied to
\eqref{eq:step-recurrence}, now gives the integrality, the stated
divisibility test, and the description of all multipliers.
\end{proof}

For $(a,b)=(0,0)$ the formula would involve $0/0$.  In this case all
$M\geq2$ are admissible and all $M$ residues are multipliers.

\begin{normalizationremark}[The multiplier on the reduced modulus]
Under the hypotheses of Theorem~\ref{thm:modular}, put $d=\gcd(M,h)$,
$m=M/d$, and $V_k=W_{r+kn}/h$.  The sampled recurrence shows that $h$
divides every term, and $\gcd(V_0,V_1)=1$.  Since $\gcd(h/d,m)=1$, for
each integer $q$ and each $k\geq0$ we have
\[
 W_{r+(k+1)n}\equiv qW_{r+kn}\pmod M
 \quad\Longleftrightarrow\quad V_{k+1}\equiv qV_k\pmod m.
\]
Suppose $M$ is admissible and $m>1$.  A prime dividing both $V_0$ and
$m$ would also divide $V_1$, so $V_0$ is a unit modulo $m$.  The first
transition therefore determines the unique reduced multiplier
\[
 q\equiv V_1 V_0^{-1}\pmod m .
\]
It too is a unit.  Indeed, substituting $V_1\equiv qV_0$ and
$V_2\equiv q^2V_0$ into \eqref{eq:step-recurrence} gives
$q(L_n-q)\equiv(-1)^n\pmod m$.  The $d$ lifts of this one residue are
exactly the multipliers modulo $M$.  For an odd prime power $m$, the
normalized sequence thus has restricted period one in the sense above.
If $m=1$, then $M\mid h$, and every residue modulo $M$ works.
\end{normalizationremark}

\begin{corollary}[Fibonacci and Lucas numbers]\label{cor:two-moduli}
For $n\geq1$, $r\geq0$ and $g=\gcd(r,n)$, the integers governing the two
classical sequences are
\begin{equation}
 K_F(r,n)=\frac{F_n^2}{F_g},\qquad
 K_L(r,n)=\frac{5F_n^2}{\gcd(L_r,L_{r+n})}.
 \label{eq:two-moduli}
\end{equation}
\end{corollary}

\begin{proof}
The initial pairs $(0,1)$ and $(2,1)$ give $\Delta_F=-1$ and $\Delta_L=5$.
For Fibonacci numbers, strong divisibility gives
$\gcd(F_r,F_{r+n})=F_{\gcd(r,r+n)}=F_g$ for $r>0$
\cite[p.~27]{McDaniel1991}.  At $r=0$ the same equality reads
$\gcd(0,F_n)=F_n$.  Substitution in~\eqref{eq:general-modulus} gives the
two formulas.
\end{proof}

The Lucas denominator can also be read from the indices alone.  With
$g=\gcd(r,n)$, the classical Lucas gcd formula~\cite[p.~27]{McDaniel1991}
gives
\[
 \gcd(L_r,L_{r+n})=
 \begin{cases}
 L_g, & n/g\ \text{even},\\
 2,   & n/g\ \text{odd and }3\mid g,\\
 1,   & n/g\ \text{odd and }3\nmid g.
 \end{cases}
\]
Indeed, for $r>0$ the coprime integers $r/g$ and $(r+n)/g$ are both odd
exactly when $n/g$ is even.  Otherwise the gcd divides $2$, and $L_j$
is even exactly when $3\mid j$, by the recurrence modulo $2$.  This last
observation also covers $r=0$, since $L_0=2$.

Whenever $K_W\geq2$, it is the largest admissible modulus.  If $K_W=1$,
none is admissible.  For example, let $(a,b)=(-2,6)$, $r=0$, $n=3$.
The sampled terms begin $-2,10,38$, and $\Delta_W=-44$, $h=2$, $K_W=88$.
At the largest modulus both $39$ and $83$ work: their common class is
$q\equiv-5\pmod{44}$.

At step $n=3$, the starting index affects the answer even for the same
modulus.  For $r=1$ one finds $K_F(1,3)=4$ and the residues
$1,3,1,3,\ldots$ modulo $4$ have multiplier $3$.  For $r=3$ the value
is $K_F(3,3)=2$, and the residues modulo $4$ are $2,0,2,0,\ldots$.
Here $4$ divides the determinant in~\eqref{eq:sample-det}, but the
transition from $0$ to $2$ is impossible.  This example explains the
divisor $h$ in~\eqref{eq:det-criterion}.  For the Lucas sequence at the
same step, $K_L(1,3)=20$ and $1,7,29,123,\ldots$ have multiplier $7$
modulo $20$.  At $r=0$, $K_L(0,3)=10$.  Modulo $10$ we have $h=2$,
$m=5$ and $(V_0,V_1)=(1,2)$.  The unique multiplier $2$ modulo $5$
lifts to $2$ and $7$ modulo $10$, and only the latter is a unit modulo $10$.

\paragraph{Computational verification.}
The ancillary Python script \texttt{verify\_congruences.py} generates the
terms by recurrence and tests every residue $q$ on the first two
transitions, followed by four more transitions for each surviving $q$.
The $93{,}744$ Fibonacci and Lucas cases cover $1\leq n\leq24$,
$0\leq r\leq30$, $2\leq M\leq64$.  A further $368{,}550$ cases cover
all $-4\leq a,b\leq4$, $1\leq n\leq10$, $0\leq r\leq12$, $2\leq M\leq36$,
including the zero sequence.  On this second grid the computed sets of
multipliers agree exactly with the classes in Lemma~\ref{lem:modular} and
the zero-sequence case.  For the $364{,}000$ cases with nonzero initial
pair, the script also enumerates the multipliers of the normalized sequence
modulo $m$ and checks their lifts modulo $M$.  The Zenodo reproducibility
archive contains the script, its expected output, and the manuscript
sources~\cite{MantovanelliZenodo}.

\section{Conclusion}
Quadratic reduction recovers Lucas's multiple-index identities and gives the
recurrence on each arithmetic progression.  An arithmetic subsequence of a
Fibonacci recurrence is geometric modulo $M$ if and only if its first two
transitions admit the same multiplier.
The determinant of the first three terms, divided by the gcd of the first
two, decides exactly when this is possible.  For each admissible modulus,
dividing the terms by this gcd gives a primitive pair and determines the
multiplier uniquely on the reduced modulus.  Its lifts give all
multipliers on the original modulus.

\paragraph{Revision note.}
Version 3 expands the historical references for the multiple-index
identities and adds the classification of geometric subsequences modulo an
integer for all integral initial values.

\section*{Acknowledgments}

\noindent\textbf{AI disclosure.}
The author used OpenAI’s ChatGPT for help in language editing,
LaTeX restructuring, mathematical explorations, and the preparation of verification
scripts. All mathematical statements, proofs, computations, references, and the final
presentation were independently verified by the author, who assumes full responsibility
for the content.

\begingroup
\small
\bibliography{references}
\endgroup

\end{document}